\documentclass[11pt]{amsart}
\usepackage{amssymb,amsmath}

\def\Bbb{\mathbb}
\def\Cal{\mathcal}
\def\Dt{\partial_t}

\def\eb{\varepsilon}

\def\R {\mathbb{R}}

\def\<{\left<}
\def\>{\right>}

\def\Dx{\Delta_x}
\def\({\left(}
\def\){\right)}

\newtheorem{proposition}{Proposition}[section]
\newtheorem{theorem}[proposition]{Theorem}
\newtheorem{corollary}[proposition]{Corollary}
\newtheorem{lemma}[proposition]{Lemma}

\theoremstyle{definition}
\newtheorem{definition}[proposition]{Definition}

\newtheorem{remark}[proposition]{Remark}

\numberwithin{equation}{section}

\def \no#1#2#3 {{\bf #1} (#3), #2.}
\def \eds#1#2#3 {#1, #2, #3.}

\title[Bi-Lipschitz Man\'e projectors for cGL equation] { Bi-Lipschitz Man\'e projectors and finite-dimensional
reduction for complex Ginzburg-Landau equation}
\author[Anna Kostianko]{Anna Kostianko${}^{1,2}$}
\address{${}^1$
University of Surrey, Department of Mathematics,
Guildford, GU2 7XH, United Kingdom.}
\address{${}^2$ \phantom{e}School of Mathematics and Statistics, Lanzhou University, Lanzhou  \\ 730000,
P.R. China}
\email{anna.kostianko@surrey.ac.uk}

\begin{document}
\begin{abstract} We present a new method of establishing the finite-dimen-\\sionality of limit dynamics (in terms of bi-Lipschitz Man\'e projectors) for semilinear parabolic systems with cross diffusion terms and illustrate it on the model example of 3D complex Ginzburg-Landau equation with periodic boundary conditions. The method combines the so-called spatial-averaging principle invented by Sell and Mallet-Paret with temporal averaging of rapid oscillations which come from cross-diffusion terms.
\end{abstract}
\maketitle
\tableofcontents
\section{Introduction}\label{s0}

 It is believed that the long-time dynamics generated by a dissipative PDE is effectively finite-dimensional,
  i.e., despite the infinite-dimensionality of the initial phase space, it can be governed by finitely
   many parameters (the so-called order parameters in the terminology of I. Prigogine) and the associated
   system of ODEs (the so-called inertial form (IF)) which describes the evolution of these order
    parameters. However, despite  many efforts made in this direction during the last 50 years (see \cite{BV,Tem1,SY,MZ,Rob2} and references therein), the precise mathematical meaning of this reduction remains unclear and requires further investigation.

 The most popular approach to study the dissipative dynamics is related to the concept of a global attractor which is by definition a compact invariant set attracting the images of all bounded sets of the phase space as time tends to infinity. Thus, on the one hand, the global attractor (if exists) captures all non-trivial dynamics of the system considered and, on the other hand, it is usually essentially smaller than the initial phase space $\Phi$  and this justifies the desired reduction of the number of degrees of freedom. Moreover, one of the main results of the attractors theory claims that, under relatively weak assumptions on the dissipative PDE considered, the global attractor exists and possesses finite Hausdorff and box-counting dimensions. In particular, it is true for 2D Navier-Stokes system, various types of reaction-diffusion equations, damped wave and Schr\"odinger equations, Ginzburg-Landau equations and many other important classes of PDEs, see \cite{BV,Tem1,Rob1}.

  If the finite-dimensional global attractor $\Cal A$ is constructed, then the Man\'e projection theorem ensures us that a projector $P$ to a generic fi\-ni\-te-di\-men\-sio\-nal plane of the phase space $\Phi$ of the problem considered is injective on $\Cal A$ if the dimension of the plane is large enough. Thus, the dynamical semigroup $S(t):\Cal A\to\Cal A$ generated by the considered PDE on the attractor is conjugated to the projected semigroup $\tilde S(t)=PS(t)P^{-1}$ acting on a finite-dimensional compact set $\tilde {\Cal A}:=P\Cal A$ and this gives us a {\it finite-dimensional} reduction. In addition, slightly more delicate arguments give us also the IF as a system of ODEs acting on this plane. Projectors which satisfy the injectivity property on the attractor are usually referred as Man\'e projectors, see \cite{Rob1,Rob2} for more details.
\par
   However, the described approach has essential drawbacks which prevent to consider it to be as a reasonable way to justify   the finite-dimensional reduction in dissipative PDEs. One of the key questions here is the smoothness of the obtained reduced semigroup $\tilde S(t)$ and the corresponding IF. It is well-known that in general the Man\'e projector can be chosen in such a way that $P^{-1}$ is H\"older continuous (in the case of abstract semilinear parabolic equations the H\"older exponent may be chosen arbitrarily close to one, see \cite{PR,Rob2,FO}). This leads to H\"older continuous reduced semigroups and IF with H\"older continuous vector fields. In contrast to this, Lipschitz (or even log-Lipschitz) continuity of inverse Man\'e projectors is much more delicate and in general the answer on the existence of Lipschitz Man\'e projectors is negative even in the class of abstract semilinear parabolic problems, see \cite{EKZ,Z} and references therein.
\par
Indeed, let us consider a semilinear parabolic equation of the form
\begin{equation}\label{0.abs}
\Dt u+Au=F(u)
\end{equation}
in a Hilbert space $H$. Here $A$ is a positive definite sectorial linear operator with compact
 inverse and $F:H\to H$ is a given non-linearity which is assumed to be bounded and at least Lipschitz
  continuous. Then, the existence of a compact global attractor $\Cal A$ with finite box-counting
  dimension is well-known, but there are examples where the attractor $\Cal A$ cannot be
   embedded into any Lipschitz or even log-Lipschitz finite-dimensional manifold and, by this reason,
    the Lipschitz or Log-Lipschitz Man\'e projections do not exist, see \cite{EKZ,Z}. The analogous
     examples have been recently constructed for the class of 1D reaction-diffusion-advection problems
   (with local non-linearities) as well, see \cite{AZ2}. In addition, the dynamics on the attractor in
    these examples has features which hardly can be interpreted as "finite-dimensional", e.g., limit
    cycles with super-exponential rate of attraction, travelling waves in Fourier space, etc.,
     see \cite{EKZ}. At the same time, the box-counting dimension in these examples remains finite
     (and is not very large) and H\"older continuous IF exists. Such examples indicate that the
     limit dissipative dynamics may be infinite dimensional despite the finiteness of box-counting dimension
      of the corresponding attractor and motivate the increasing interest to study alternative
       constructions for the finite-dimensional reduction which are not based on
       box-counting dimension and Man\'e projection theorem.
\par
An ideal situation is the case where the considered system possesses the so-called inertial manifold (IM) which is a finite-dimensional smooth (at least $C^1$) normally-hyperbolic invariant manifold in the phase space with a global attraction property. Then, the restriction of the equation to this manifold gives the desired IF and we also have that any trajectory of the initial system is attracted exponentially to some trajectory of the IF (the so-called asymptotic phase or exponential tracking property, see \cite{F,FST,Mi,Tem1,Z} and references therein). So, this construction gives natural and transparent finite-dimensional reduction for a number of important equations such as 1D reaction-diffusion equations, Swift-Hohenberg and Kuramoto-Sivashinsky equation, etc.
\par
The existence of such a manifold requires strong separation of slow and fast variables which is
usually formulated in terms of invariant cones, see \cite{Rom,Tem1,Z} and references therein. In turn,
in order to verify these conditions, the so-called spectral gap conditions (which are much easier to check)
 are usually exploited. For instance, in the case of equation \eqref{0.abs} with self-adjoint
  (or normal) operator $A$, the spectral gap condition for existence of $N$-dimensional IM
  reads
\begin{equation}\label{0.sg}
\lambda_{N+1}-\lambda_N>2L,
\end{equation}
where $\{\lambda_n\}_{n=1}^\infty$ are the eigenvalues of $A$ enumerated in the non-decreasing order
 and $L$ is a Lipschitz constant of the non-linearity $F$. However, these conditions are very restrictive,
  for instance, for the most natural case where $A$ is a Laplacian in a bounded domain, they are satisfied
   for general non-linearities  in 1D case only.
\par
It is also known that, in the case where $F$ is a general non-linearity, the spectral gap conditions \eqref{0.sg} are sharp in the sense that if they are violated for all $N$, we always can construct a nonlinearity $F$ such that \eqref{0.abs} will not possess any finite-dimensional IM, see \cite{EKZ,Z}. In contrast to this, for concrete particular classes of equations \eqref{0.abs} the IM may exist even in the case where the spectral gap conditions are violated.
\par
Up to the moment, there are two approaches to build up IMs beyond of spectral gap conditions. The first one is to try to make a change of the dependent variable transforming the equation to a new one for which the spectral gap conditions are satisfied or/and to embed it to a larger system of equations with spectral gap conditions. This works, e.g., for 1D reaction-diffusion-advection problems with Dirichlet and Neumann boundary conditions (surprisingly, in the case of periodic boundary conditions an IM may not exist, see \cite{AZ1,AZ2}). This approach is somehow inspired by the attempt to get the IM for 2D Navier-Stokes equation via the so-called Kwak transform (\cite{Kwak,Tem}) which unfortunately contains an irrecoverable  error, see \cite{AZ4} for more details.
\par
The second one is the so-called spatial averaging method which works mainly for 2D and 3D tori and is related with the fact that the multiplication operator on a smooth function $f(x)$ restricted to the properly chosen "intermediate" modes  is close to the multiplication on its spatial average $\<f\>$.
This approach has been initially developed by Sell and Mallet-Paret to build up IMs for scalar reaction-diffusion equations on 2D and 3D tori, see \cite{MPS} and is extended nowadays to many other classes of equations, e.g., for 3D Cahn-Hilliard equation (see \cite{AZ5}) or modified Navier-Stokes equations (see \cite{A}). Note that this method usually does not work for systems since it is crucial that $\<f\>$ is a scalar, not a matrix (and exceptional case is exactly the modified Navier-Stokes system where $\<f\>$ equals to zero identically).
\par
An intermediate step between H\"older continuous IF build via the Man\'e projection theorem and IMs is the
 so-called Romanov theory which gives necessary and sufficient conditions for the existence of
  Lipschitz continuous Man\'e projections and Lipschitz IFs, see \cite{Rom1,Rom4}. The conditions for that
  are somehow close but slightly weaker than the ones for the IMs. For instance, the cone condition
  also plays a crucial role here, but it should be verified on the global attractor only, not in the
  whole phase space. This allows us to verify it locally for the linearization of our equation
  at every complete bounded trajectory belonging to the attractor without taking care about cut-off
   procedures (as known the proper cut-off of the considered equation is one of the key technical
    problems in both approaches to IMs mentioned above, see \cite{AZ1,AZ2,MPS,Z} for more details).
     On the other hand we believe that the cut-off problem has a technical nature, so in more or
 less general situation the existence of Lipschitz Man\'e projections should imply also the existence
 of an IM. By this reason we treat establishing the existence of Lipschitz Man\'e projectors as
 the most essential step in constructing the IM. In addition, at this step we may demonstrate key
 ideas in a more transparent way avoiding the technicalities related with the cut-off procedure.
\par
The main aim of this paper is to present a new method of verifying the existence of Lipschitz Man\'e projectors and potentially IMs which we refer as {\it spatio-temporal averaging} method. This method is illustrated on the model example of 3D complex Ginzburg-Landau equation with periodic boundary conditions or more general, the following cross-diffusion system:
\begin{equation}\label{0.cross}
\Dt \Psi=(1+i\omega)\Dx \Psi-f(\Psi,\bar \Psi)
\end{equation}
endowed with periodic BC. Here $\Psi=\Psi_{Re}(t,x)+i\Psi_{Im}(t,x)$ is an unknown complex-valued function, $\omega\in\R$, $\bar \Psi=\Psi_{Re}-i\Psi_{Im}$ is a complex conjugate function and $f$ is a given smooth function. In the particular case
$$
f(\Psi,\bar\Psi)=(1+i\beta)\Psi|\Psi|^2-(1+i\gamma)\Psi
$$
we end up with the classical Ginzburg-Landau equation (see \cite{DG,Mie} and references therein for more details concerning this equation and its physical meaning).
\par
The suggested method is a combination of the spatial averaging principle of Sell and Mallet-Paret with the classical temporal averaging for equations with large dispersion (in the spirit of  \cite{ATZ}). Roughly speaking, at the first step we use spatial averaging in order to get rid of the dependence on spatial variable $x$ in the equation of variations and replace $f'(\cdot)v$ by $\<f'\>v$. However, the matrix $\<f'\>$ is not a scalar matrix, so this is not enough to get the result. The key observation here is that if the cross-diffusion coefficient $\omega\ne0$, the term $i\omega\Dx v$ produces a {\it large} dispersion on the intermediate modes (no matter how small $\omega$ is) which can be averaged. Performing this temporal averaging, we finally arrive at a scalar matrix which allows us to complete the arguments, see section \ref{s1}.
\par
The main result of the paper is the following theorem.
\begin{theorem}\label{Th0.main} Let $\omega\ne0$ and let the nonlinearity $f$ be smooth. Assume also that equation \eqref{0.cross} is globally solvable in the phase space $\Phi=H^2_{per}(\Bbb T^3)$ and possesses a dissipative estimate
\begin{equation}\label{0.dis}
\|\Psi(t)\|_{\Phi}\le Q(\|\Psi(0)\|_{\Phi})e^{-\alpha t}+C_*,\ \ t\ge0,
\end{equation}
where the monotone increasing function $Q$ and positive constants $\alpha$ and $C_*$ are independent of $\Psi(0)$. Then the corresponding solution semigroup in the phase space $\Phi$ possesses
 a global attractor $\Cal A$ which has a Lipschitz continuous Man\'e projector. In particular, equation \eqref{0.cross} possesses an IF with Lipschitz continuous vector field.
\end{theorem}
The proof of this theorem is given in section \ref{s2}, see also Remark \ref{Rem2.cond} for more details on validity of the dissipative estimate \eqref{0.dis}. Note also that the assumption $\omega\ne0$ is crucial here. The counterexamples to existence of a normally hyperbolic IM in the self-adjoint case $\omega=0$ are given in \cite{Rom3}.

\section{Key estimates for the linearized equation}\label{s1}
In this section, we study backward in time solutions for the following linear complex Ginzburg-Landau equation
\begin{equation}\label{3.lincgl}
\Dt v-(1+i\omega)\Dx v+a(t,x)v+b(t,x)\bar v=h(t),\ \ \ t\le0
\end{equation}
in a domain $\Omega=\mathbb T^3:=(-\pi,\pi)^3$ endowed by periodic boundary conditions. Here $v(t,x)=v_r(t,x)+iv_i(t,x)$ is an unknown complex valued function, $\bar v=v_r-iv_i$ is a complex conjugation, $\omega\in\R$, $\omega\ne0$, is a given real number, $a$ and $b$ are given functions which satisfy
\begin{equation}\label{3.bounded}
\|a\|_{C^1_b(\R\times\mathbb T^3)}+\|b\|_{C^1_b(\R\times\mathbb T^3)}\le K,
\end{equation}
and $h(t)$ is a given function the conditions on which will be specified later.
\par
We want to solve problem \eqref{3.lincgl} backward in time with an extra initial condition
\begin{equation}\label{3.in}
P_Nv\big|_{t=0}=v_+
\end{equation}
in the proper weighted spaces. Here and below $P_N: H:=L^2(\Omega)\to H_N$ is an orthoprojector to the finite-dimensional subspace $H_N$ generated by all eigenvectors $e_n$ of the Laplacian $-\Dx$ (with periodic boundary conditions) which eigenvectors $\lambda_n$ satisfy $\lambda_n\le N$.
\par
We first note that without loss of generality, we may assume that
\begin{equation}\label{3.mean}
\<a(t)\>:=\frac1{(2\pi)^3}\int_{\Omega}a(t,x)\,dx\equiv0.
\end{equation}
Indeed, if this condition is violated, we may change the dependent variable
\begin{equation}\label{2.trans}
w(t)=e^{\int_0^t\<a(s)\>\,ds}v(t)
\end{equation}
which gives
\begin{multline}\label{3.eqmean}
\Dt w-(1+i\omega)\Dx w+(a(t,x)-\<a(t)\>)w+b(t,x)e^{2i\int_0^t\<a_i(s)\>\,ds}\bar w=\\=e^{\int_0^t\<a(s)\>\,ds}h(t):=\tilde h(t).
\end{multline}
We see that the new weight satisfies
\begin{equation}\label{2.weight}
e^{-Kt}\le \big |e^{\int_0^t\<a(s)\>\,ds}\big|\le e^{Kt}
\end{equation}
independently of the choice of $N$ and new coefficients $a$ and $b$ satisfy \eqref{3.mean} and inequality  \eqref{3.bounded} (maybe with new constant $K'$ depending only on $K$). By this reason, we assume from the very beginning that \eqref{3.mean} is satisfied.
\par
\begin{theorem}\label{Th2.main} Let the assumptions \eqref{3.bounded} and \eqref{3.mean} hold. Then, there exists an infinite number of $N$s such that, for every
$$
h\in \mathcal H^-_\theta:=L^2_{e^{\theta t}}(\R_-,H),\ \ \theta=N+\frac12
$$
and every $v_+\in H_N$,
problem \eqref{3.lincgl}, \eqref{3.in} possesses a unique solution $v\in\mathcal H^-_\theta$ and the following estimate holds:
\begin{equation}\label{3.main}
\|v\|_{\mathcal H^-_\theta}\le C\(\|h\|_{\mathcal H^-_\theta}+\|v^+\|_H\).
\end{equation}
Moreover, the sequence of $N$s and the constant $C$ depend only on the constant $K$ in assumption \eqref{3.bounded} and are independent of the concrete choice of $a$ and $b$ satisfying this assumption.
\end{theorem}
\begin{proof} We divide it on several steps.
\par
{\it Step 1. Elementary transformations.} First we reduce the problem to the non-weighted case by the standard  change of variables:
$$
w(t)=e^{\theta t}v.
$$
This gives
\begin{equation}\label{3.nw}
\Dt w-(1+i\omega)\Dx w-\theta w+a w+b\bar w=e^{\theta t}h(t):=\tilde h(t),\ \ P_Nw\big|_{t=0}=v^+.
\end{equation}
Thus, instead of proving weighted estimate \eqref{3.main}, it is equivalent to verify its  non-weighted analogue (in the space $\mathcal H^-_0$) for equation \eqref{3.nw}.
\par
Next, we get rid of the initial data $v^+$. To this end we consider  equation \eqref{3.nw} in the particular case $a=b=0$:
\begin{equation}
\Dt w-(1+i\omega)\Dx w -\theta w=\tilde h(t),\ \ P_Nw\big|_{t=0}=v^+
\end{equation}
and split it into the Fourier series with respect to the eigenvectors $\{e_n\}_{n=1}^\infty$ of the Laplacian: $w(t)=\sum_{n=1}^\infty w_n(t)e_n$. Then the Fourier coefficients $w_n(t)$ solve
\begin{equation}\label{3.f0}
\frac d{dt}w_n+(\lambda_n-\theta+i\omega\lambda_n)w_n=\tilde h_n,\ t\le0, \ w_n(0)=v_n^+, \text{ if } \lambda_n\le N,
\end{equation}
 where $\tilde h_n$ and $v^+_n$ are the Fourier coefficients of $\tilde h$ and $v^+$ respectively.
To proceed further, we need the following simple lemma which is one of the key technical tools for proving the theorem.

\begin{lemma}\label{Lem3.key} For every, $v^+\in H_N$ and every $\tilde h_n\in L^2(\R_-,\mathbb C)$, problem \eqref{3.f0} possesses a unique solution $w_n\in  L^2(\R_-,\mathbb C)$ and the following estimate holds:
\begin{equation}
\|w_n\|_{L^2(\R_-,\mathbb C)}\le \frac 1{|\lambda_n-\theta|}\|\tilde h_n\|_{L^2(\R_-,\mathbb C)}+\frac1{\sqrt{2|\lambda_n-\theta|}}\|v^+_n\|_{\mathbb C}.
\end{equation}
\end{lemma}
\begin{proof}[Proof of the lemma] We recall   that $\theta=N+\frac12$. By this reason, the explicit solution
$$
w_n(t)=e^{(\theta-\lambda_n-i\omega\lambda_n)t}v_n^+ \text{ if } \lambda_n\le N
$$
of equation \eqref{3.f0} with $\tilde h_n=0$ belongs to $L^2(\R_-,\Bbb C)$ and satisfies
$$
\|w_n\|_{L^2(\R_-,\Bbb C)}\le \frac1{\sqrt{2|\lambda_n-\theta|}}|v_n^+|.
$$
Thus, we only need to verify the estimate for the case $v_n^+=0$ (if $\lambda_n\le N$). It is not difficult to check (see \cite{Z}) that the solution $w_n$ can be found as a unique solution of
$$
\frac{d}{dt}w_n+(\lambda_n-\theta+i\lambda_n)w_n=\tilde h_n
$$
defined for all $t\in\R$ and belonging to $L^2(\R,\Bbb C)$ (where we extend $\tilde h_n$ by zero for $t\ge0$. After that we may do Fourier transform in time and use Plancherel equality to get
$$
\|w_n\|_{L^2(\R,\Bbb C)}\le\frac1{|\lambda_n-\theta|}\|\tilde h_n\|_{L^2(\R,\Bbb C)}
$$
and the lemma is proved.
\end{proof}
Using this lemma, we construct a solution $W=W(v^+)$ of the problem
$$
\Dt W-(1+i\omega)\Dx W-\theta W=0,\ \ P_NW\big|_{t=0}=v^+
$$
belonging to $L^2(\R_-,H)$ and satisfying
$$
\|W\|_{\mathcal H_0^-}\le \|v^+\|_H
$$
(here we have used that $\lambda_n\in\Bbb Z$ and therefore $|\lambda_n-\theta|\ge\frac12$). Introducing now $\tilde w=w-W$, we see that this function solves
$$
\Dt\tilde w-(1+i\omega)\Dx\tilde w-\theta\tilde w+a\tilde w+b\bar{\tilde w}=h_1(t), \ \ P_N\tilde w\big|_{t=0}=0
$$
where $h_1(t):=\tilde h(t)+a W(t)+b\bar W(t)$. Thus,
$$
\|h_1\|_{\Cal H_0^-}\le\|\tilde h\|_{\Cal H_0^-}+2K\|v^+\|_H
$$
and the new function $\tilde w\in\Cal H_0^-$ satisfies equation \eqref{3.nw}, but already with $v^+=0$. By this reason, we may assume that $v^+=0$ from the very beginning and study problem \eqref{3.nw} with $v^+=0$ only.
\par
{\it Step 2. Reduction to intermediate modes.} For any $0<L<N$, $L\in\Bbb N$, let us introduce the orthoprojectors $\Cal P_{N,L}$, $\Cal I_{N,L}$ and $\Cal Q_{N,L}$ to the lower Fourier modes ($\{e_n\}$ with $\lambda_n<N-L$), intermediate modes (with $N-L\le\lambda_n\le N+L$) and higher modes (with $\lambda_n>L+N$) respectively. At this moment $L$ may be arbitrary, but later we will use these projectors in the situation where
\begin{equation}\label{3.range}
0<K\ll L\ll N.
\end{equation}
We split the solution $w$ of equation \eqref{3.nw} (from now on  we always assume that $v^+=0$) in a sum of 3 components:
$$
w(t)=\Cal P_{N,L}w(t)+\Cal I_{N,L}w(t)+\Cal Q_{N,L}w(t):=w_+(t)+z(t)+w_-(t).
$$
Then applying the projectors to \eqref{3.nw}, we get
\begin{equation*}
\begin{gathered}
\Dt w_+-(1+i\omega)\Dx w_+-\theta w_+=\\=\tilde h_+-
\Cal P_{N,L}a(w_++w_-+z)-\Cal P_{N,L}b(\bar w_++\bar w_-+\bar z),\\
\Dt w_--(1+i\omega)\Dx w_--\theta w_-=\\=\tilde h_--
\Cal Q_{N,L}a(w_++w_-+z)-\Cal Q_{N,L}b(\bar w_++\bar w_-+\bar z),\\
\Dt z-(1+i\omega)\Dx z-\theta z=\\=\tilde h_0-
\Cal I_{N,L}a(w_++w_-+z)-\Cal I_{N,L}b(\bar w_++\bar w_-+\bar z),\\
\end{gathered}
\end{equation*}
where $\tilde h=\tilde h_++\tilde h_0+\tilde h_-$ is a splitting of $\tilde h$ to lower, intermediate and higher modes.
\par
Let us try to solve the first and the second equations of this system assuming that $z\in L^2(\R_-,H)$ is given. To this end, we need the following lemma.

\begin{lemma}\label{Lem3.ni} Let $h_\pm\in L^2(\R_-,H_{\pm})$ be given. Then the equation
$$
\Dt w_\pm-(1+i\omega)\Dx w_\pm-\theta w_\pm=h_\pm,\ \ \Cal P_{N,L} w_\pm\big|_{t=0}=0
$$
possesses a unique solution $w_\pm\in L^2(\R_-,H_\pm)$ and the following estimate holds:
\begin{equation}\label{3.l-h}
\|w_{\pm}\|_{L^2(\R_-,H_\pm)}\le \frac {1}L\|h_{\pm}\|_{L^2(\R_-,H_{\pm})}.
\end{equation}
Here and below $H_+=\Cal P_{N,L}H$, $H_I:=\Cal I_{N,L}H$ and $H_-=\Cal Q_{N,L}H$.
\end{lemma}
Indeed, this result is a straightforward corollary of Lemma \ref{Lem3.key} and the fact that $|\lambda_n-N-\frac12|>L$ if $\lambda_n$ does not belong to the intermediate modes.
\par
The last lemma allows us to solve uniquely equations for $w_+$ and $w_-$ if the intermediate component $z$ is given and $K\ll L$. Indeed, to this end, we just need to invert the left-hand sides of equations for $w_+$ and $w_-$ and use the Banach contraction theorem (the contraction will be guaranteed by estimates \eqref{3.l-h} and \eqref{3.bounded}. This gives the following result.

\begin{lemma}\label{Lem3.Lyap} Let $K\ll L$ and let $z\in L^2(\R_-,H_I)$ be given. Then, there are bounded linear operators
$$
\Phi_\pm:L^2(\R_-,H_I)\to L^2(\R_-,H_\pm),\ \ \Psi_\pm: L^2(\R_-,H_\pm)\to L^2(\R_-,H_\pm)
$$
such that the unique solutions $w_\pm\in L^2(\R_-,H_\pm)$ for the lower and higher modes are given by
$$
w_\pm=\Phi_\pm z+\Psi_\pm\tilde h_\pm.
$$
Moreover, the following estimates hold:
\begin{equation}
\|\Phi_\pm\|_{\Cal L(L^2(\R_-,H_I),L^2(\R_-,H_\pm))}+\|\Psi_\pm\|_{\Cal L(L^2(\R_-,H_\pm),L^2(\R_-,H_\pm))}\le C\frac KL
\end{equation}
where the constant $C$ is independent of $N$, $K\ll L$, $L$ and the choice of $a$ and $b$.
\end{lemma}
This lemma allows us to express the functions $w_\pm$ through the intermediate function $z$ and put these expressions back to the equation for intermediate modes $z$. This gives us the following result.

\begin{lemma}\label{Lem3.red} Let $K\ll L$. Then equation \eqref{3.nw} is equivalent to the following non-local in time equation:
\begin{equation}\label{3.red}
\Dt z-(1+i\omega)\Dx z-\theta z+\Cal I_{N,L}az+\Cal I_{N,L}b\bar z=\Phi z+g,
\end{equation}
where the function $g=g(\tilde h)$ satisfies
$$
\|g\|_{L^2(\R_-,H_I)}\le C(K+1)\|\tilde h\|_{L^2(\R_-,H)}
$$
and the linear bounded operator $\Phi:L^2(\R_-,H_I)\to L^2(\R_-,H_I)$ possesses the following estimate:
\begin{equation}\label{3.small}
\|\Phi\|_{\Cal L(L^2(\R_-,H_I),L^2(\R_-,H_I))}\le C\frac{K^2}L,
\end{equation}
where the constant $C$ is independent of $N$, $L$ and $K$.
\end{lemma}
As we will see later, the numbers $N$ and $L$ are actually in our disposal, so we may fix $L$ to be large enough and then the non-local term $\Phi z$ will be arbitrarily small. Thus, the proof of the theorem is mainly reduced to solving finite-dimensional equation \eqref{3.red} (with $\Phi=0$) for the intermediate modes. However, this equation is still complicated since the operators $\Cal I_{N,L}a$ and $\Cal I_{N,L}b$ couple all intermediate modes. So, more steps are necessary.
\par
{\it Step 3. Spatial averaging.} At this stage we get rid of the dependence of the coefficients $a$ and $b$ on $x$ using the so-called spatial averaging principle used in \cite{MPS} for constructing the inertial manifolds for 3D scalar reaction-diffusion equations, see also \cite{ACZ,AZ5} for further development and more applications of this method. The key technical tool of this method is the following lemma.

\begin{lemma}\label{Lem3.s-av} Let $\phi\in C^1(\Bbb T^3)$ satisfy $\|\phi\|_{C^1}\le K$. Then, for every $\eb>0$, $K>0$ and  $L>0$ there is an infinite sequence of $N$s such that
\begin{equation}\label{3.avs}
\|\Cal I_{N,L}\phi\Cal I_{N,L}v-\<\phi\>\Cal I_{N,L}v\|_{L^2}\le\eb\|v\|_{L^2},\ \ v\in L^2(\Bbb T^3).
\end{equation}
The sequence of $N$s depends only on $\eb$, $K$ and $L$ (and is independent of the concrete choice of $\phi$).
\end{lemma}
The proof of this lemma is based on the number theoretic results about integer points in a spherical layers and elementary harmonic analysis and can be found in \cite{MPS}, see also \cite{Z}. Note also that the assumption $\phi\in C^1$ can be replaced by $\phi\in C^\kappa$ for some $\kappa>0$.
\par
Applying this lemma to the terms $\Cal I_{N,L}az$ and $\Cal I_{N,L}b\bar{z}$, we get the following result.

\begin{lemma}\label{Lem3.rred} For every $\eb>0$ and $K>0$ there exists a sequence of $L$s and $N$s, $L\ll N$ (e. g., $L<\eb^2N$) such that equation \eqref{3.red} is equivalent to
\begin{equation}\label{3.rred}
\Dt z-(1+i\omega)\Dx z-\theta z+\beta(t)\bar z=\Phi^\eb(z)+g,\  \ z\in L^2(\R_-,H_I),
\end{equation}
where $\beta(t):=\<b(t\>$, the linear operator $\Phi^\eb: L^2(\R_-,H_I)\to L^2(\R_-,H_I)$ satisfies
\begin{equation}\label{3.ebs}
\|\Phi^\eb\|_{\Cal L(L^2(\R_-,H_I),L^2(\R_-,H_I))}\le\eb
\end{equation}
and the norm of $g$ is independent of $L$, $\eb$ and $N$.
\end{lemma}
\begin{proof}[Proof of the lemma] Indeed, applying Lemma \ref{Lem3.s-av} to the term $\Cal I_{N,L}a z$ and using that $\Cal I_{N,L}z=z$ and the assumption $\<a(t)\>=0$, we see that this term is actually of order $\eb$ (we include this corrector to $\Phi^\eb$). Analogously, applying Lemma \ref{Lem3.s-av} to the second term $\Cal I_{N,L}b\bar z$, we get the term $\beta(t)\bar z$ plus small corrector which is included to $\Phi^\eb$. Finally, the term $\Phi(z)$ can be made of order $\eb$ by the choice of $L$ (due to estimate \eqref{3.small}). Thus, the lemma is proved.
\end{proof}
Equation \eqref{3.rred} now can be split to a finite number of 2nd order ODEs coupled through the small perturbation $\Phi^\eb$ only. Indeed, decomposing $z$ into Fourier series, we get
\begin{equation}\label{3.rsplit}
\frac d{dt}z_n+(\lambda_n-\theta)z_n+i\omega\lambda_n z_n+\beta(t)\bar z_n=\Phi_n^\eb(z)+g_n
\end{equation}
for all $n\in\Bbb N$ such that $N-L\le\lambda_n\le N+L$. Here and below $\Phi_n^\eb$ and $g_n$ are Fourier components of $\Phi^\eb$ and $g$ respectively. However, the extra term $\beta(t)\bar z$ still does not allow us  to do standard estimates and we need one more step to handle~it.
\par
{\it Step 4. Temporal averaging.} Equations \eqref{3.rsplit} contain the {\it large} dispertion term $i\omega_n z_n$ with $\omega_n:=\omega\lambda_n$. Here we have crucially used the assumption $\omega\ne0$ and the fact that $\lambda_n\in[N-L,N+L]$ where $N$ is big and $L\ll N$. By this reason, it looks natural to utilize the rapid in time oscillations caused by this dispersive term. To this end, following, say, \cite{ATZ} (see also references therein), we do the change of variables
$$
Z_n(t)=e^{i\omega_n t}z_n(t).
$$
Then, we get
\begin{multline}\label{4.rapid}
\Dt Z_n+(\lambda_n-\theta)Z_n+e^{2 i\omega_n t}\beta(t)\bar Z_n=\\=e^{i\omega_n t}\Phi_n(\{e^{-i\omega_k t}Z_k\})+e^{i\omega_n t}g_n:=\Phi_n^\eb(Z)+G_n.
\end{multline}
Since our transform is an isometry in $H_I$, we have
\begin{multline}
\|Z\|_{L^2(\R_-,H_I)}=\|z\|_{L^2(\R_-,H_I)},\\ \|G\|_{L^2(\R_-,H_I)}=\|g\|_{L^2(\R_-,H_I)},\ \|\Phi^\eb(Z)\|_{L^2(\R_-,H_I)}\le\eb\|Z\|_{L^2(\R_-,H_I)}.
\end{multline}
Thus, equation \eqref{4.rapid} preserves  all good properties of equation \eqref{3.rsplit}, so it is sufficient to prove the unique solvability of \eqref{4.rapid} in the space $L^2(\R_-,H_I)$. This equation has an essential advantage since it contains an explicit rapidly oscillating term with zero mean, so by the classical averaging theory (see \cite{ATZ} and references therein), we expect that this term will be averaged to zero and the solvability of \eqref{4.rapid} for a general $\beta(t)$ should follow from the particular case $\beta=0$ (where it is obvious).
\par
Let us justify this idea. As usual, we transform \eqref{4.rapid} as follows:
\begin{multline}\label{3.huge}
\frac d{dt}\(Z_n-\frac i{2\omega_n}e^{2i\omega_n t}\beta(t)\bar Z_n\)+(\lambda_n-\theta)Z_n=\Phi_n^\eb(Z)+G_n-\\-\frac{i}{2\omega_n}e^{2i\omega_n t}\beta'(t)\bar Z_n-
\frac{i}{2\omega_n}e^{2i\omega_n t}\beta(t)\frac d{dt}\bar Z_n=\Phi_n^\eb(Z)+G_n-\\-\frac{i}{2\omega_n}e^{2i\omega_n t}\(\beta'(t) - \beta(t)(\lambda_n - \theta) \)\bar Z_n+ \frac{i}{2\omega_n}|\beta(t)|^2 Z_n -\\-\frac{i}{2\omega_n}e^{2i\omega_n t}\(\bar \Phi_n^\eb(Z)+\bar G_n\)=\tilde{\Phi}_n^\eb(Z)+\tilde G_n.
\end{multline}
We claim that the norm $\tilde{\Phi}^\eb$ remains of order $\eb$ if we take $N$ large enough and the norm of $\tilde G$ remains bounded. Indeed, $|\beta'(t)|+|\beta(t)|\le K$ is bounded. The term $\frac1{2\omega_n}\le \frac 1{2\omega(N-L)}$ can be made of order $\eb$ if $N$ is large enough. Finally, the term $\frac{|\lambda_n-\theta|}{2\omega_n}\le \frac{2 L+1}{4\omega(N-L)}$ also can be made of order $\eb$ if $N$ is large enough and $L\ll N$. Thus, the new terms $\tilde{\Phi}^\eb$ and $\tilde G$ satisfy the same good estimates as the initial terms $\Phi^\eb$ and $G$.
\par
To complete the proof, we need one more change of variables:
$$
U_n(t):=Z_n-\frac i{2\omega_n}e^{2i\omega_n t}\beta(t)\bar Z_n.
$$
The inverse transform to this is given by
$$
Z_n(t)=\frac 1{1-\frac{|\beta(t)|^2}{4\omega_n^2}}U_n(t)+\frac{\frac{i\beta(t)}{2\omega_n}e^{2i\omega_n t}}
{1-\frac{|\beta(t)|^2}{4\omega_n^2}}\bar U_n(t).
$$
Analogously to previous estimates, we see that the linear transform $U_n\to Z_n$ is invertible and is $\eb$ close to identity (if $N$ is large enough), so inserting the formula for $Z_n$ into \eqref{3.huge}, we finally arrive at
\begin{equation}\label{3.last}
\frac d{dt}U_n+(\lambda_n-\theta)U_n=\widehat{\Phi}^\eb(U)+\widehat G_n,
\end{equation}
where the norm of the operator $\widehat{\Phi}^\eb_n$ is of order $\eb$ and the norm of $\widehat G$ is uniformly bounded as $\eb\to0$. Using now Lemma \ref{Lem3.key} and the fact that $|\lambda_n-\theta|\ge\frac12$, by choosing $\eb>0$ small enough, we see that equations \eqref{3.last} are uniquely solvable in $L^2(\R_-,H_I)$ by Banach contraction theorem. This finishes the proof of the theorem.
\end{proof}
We conclude this section by the following corollary of the proved theorem which is necessary for the non-linear case.
\begin{corollary}\label{Cor1.main} Let the coefficients $a$ and $b$ satisfy condition \eqref{3.bounded} (assumption \eqref{3.mean} is not assumed). Then, there exist infinitely many $N$s and the corresponding exponents $\theta_N=\theta_N(K)$ such that  any bounded backward solution
$v\in C_b(\R_-,H)$  of the equation
\begin{equation}\label{3.lincgl1}
\Dt v-(1+i\omega)\Dx v+a(t,x)v+b(t,x)\bar v=0,\ \ \ t\le0
\end{equation}
satisfies the following estimate:
\begin{equation}\label{2.mest}
\|v(t)\|_{H}\le C_N e^{-\theta_N t}\|P_Nv(0)\|_H,\ \ \ t\le0,
\end{equation}
where the constants $C_N$ and $\theta_N$ depend only on $N$ and $K$, but are independent of the concrete choice of the solution $v$.
\end{corollary}
\begin{proof} Indeed, since $v\in C_b(\R_-,H)$,  the transform \eqref{2.trans} together with \eqref{2.weight} gives us the solution $w\in \Cal H_{K+1}$ of equation \eqref{3.eqmean} (with $h=0$) where the condition \eqref{3.mean} is satisfied. If we assume in addition that $N$ is large enough ($N>K+3/2$), we may apply Theorem \ref{Th2.main} and get the following estimate:
\begin{equation}\label{3.west}
\|w\|_{\Cal H_{N+\frac12}}\le C_N\|P_N v(0)\|_H
\end{equation}
which together with the parabolic smoothing property implies that
\begin{equation}\label{3.wc}
\|w(t)\|_H\le C_N e^{-(N+\frac12)t}\|P_Nv(0)\|_H.
\end{equation}
Returning back to the variable $v$ and using \eqref{2.weight} again, we end up with the desired estimate \eqref{2.mest} and finish the proof of the corollary.
\end{proof}
\section{The nonlinear case: finite dimensional reduction on the attractor}\label{s2}
We now study the following semi-linear cross-diffusion equation:
\begin{equation}\label{4.eq}
\Dt\Psi=(1+i\omega)\Dx \Psi +f(\Psi,\overline{\Psi}):=A\Psi+F(\Psi),\ \ \Psi\big|_{t=0}=\Psi_0
\end{equation}
in a domain $\Omega:=(-\pi,\pi)^3$ endowed with periodic boundary conditions. Here $\Psi=\Psi_r(t,x)+i\Psi_i(t,x)$ is an unknown complex valued function, $\omega\in\R$ is a given constant and $f$ is a given smooth function.
\par
We assume that this equation is globally well-posed in higher energy norms and is dissipative. To be more precise, we assume that for any $\Psi_0\in H^2=H^2(\Omega)$ equation \eqref{4.eq} possesses a unique solution $\Psi\in C([0,T],H^2)$ for all $T>0$ and the following estimate holds:
\begin{equation}\label{4.dis}
\|\Psi(t)\|_{H^2}\le Q(\|\Psi_0\|_{H^2})e^{-\alpha t}+Q_*,\ \ t\ge0,
\end{equation}
for some monotone increasing function $Q$ and positive  constants $\alpha,Q_*$ which are independent of $\Psi_0$.
\par
If this assumption is satisfied, then equation \eqref{4.eq} generates a dissipative semigroup $S(t)$, $t\ge0$ in the phase space $H^2$ via
\begin{equation}\label{4.sem}
S(t):H^2\to H^2,\ \ S(t)\Psi_0:=\Psi(t),\ \ \Psi_0\in H^2,
\end{equation}
where $\Psi(t)$ is a solution of equation \eqref{4.eq} with the initial data $\Psi_0$ at time moment $t$ and this semigroup possesses the so-called global attractor in $H^2$. We recall that a set $\Cal A$ is a global attractor for the semigroup $S(t): H^2\to H^2$ if
\par
1. The set $\Cal A$ is compact in $H^2$;
\par
2. It is strictly invariant: $S(t)\Cal A=\Cal A$ for all $t\ge0$;
\par
3. It attracts images of all bounded sets of $H^2$ when $t\to\infty$, i.e., for every bounded set $B$ in $H^2$ and every neighbourhood $\Cal O(\Cal A)$ of the set $\Cal A$, there exists $T=T(B,\Cal O)$ such that
$$
S(t)B\subset\Cal O(\Cal A)\ \ \text{for all}\ \ t\ge T,
$$
see \cite{BV,Tem1,SY} and references therein.
\par
We now state the standard theorem on the existence of a global attractor for the semigroup  \eqref{4.sem}.
\begin{theorem}\label{Th3.stand} Let the problem \eqref{4.eq} possess a unique solution $\Psi$ such that $\Psi\in C([0,T],H^2)$ for all $T>0$ and let the dissipative estimate \eqref{4.dis} be satisfied. Then the solution semigroup $S(t)$ defined by \eqref{4.sem} possesses a global attractor $\Cal A$ in $H^2$. Moreover, this attractor is smooth, i.e., it is a bounded set in $H^s$ for every $s\ge0$:
\begin{equation}\label{4.abound}
\|\Cal A\|_{H^s(\Omega)}\le C_s
\end{equation}
for all $s\ge0$.
\end{theorem}
This statement is a straightforward corollary of the classical smoothing property for semilinear parabolic equations, see \cite{BV,Tem1,Z}. The choice of the phase space $H^2$ is related with our choice of 3D case  where Sobolev embedding $H^2\subset C$ allows us to control the $L^\infty$-norm of the solution. This in turn allows to control the non-linearity without any growth restrictions. Of course, this result is not restricted to cross-diffusion equations and holds for general semilinear paraboloc equations.
\par
\begin{remark}\label{Rem2.cond} The stated theorem is a {\it conditional} result which requires the key dissipative estimate \eqref{4.dis} to be satisfied. Verification of this estimate  may be rather delicate in concrete examples and necessary and sufficient conditions for it are not known even in the case of classical complex Ginzburg-Landau (cGL) equations which correspond to
\begin{equation}\label{4.cgl}
f(\Psi,\overline{\Psi}):=(1+i\beta)\Psi-(1+i\gamma)\Psi|\Psi|^2, \ \beta,\gamma\in\R.
\end{equation}
The list of known sufficient conditions for cGL equations in terms of the parameters $\omega,\beta,\gamma$ can be found, e.g., in \cite{DG}, see also \cite{MZ} for the case of more general non-linearities. In particular, the classical cGL always possesses the dissipative estimate in the $H=L^2(\Omega)$ norm:
\begin{equation}\label{4.hdis}
\|\Psi(t)\|_H\le C\|\Psi_0\|_He^{-\alpha t}+C_*,
\end{equation}
but it is not enough to get dissipativity in higher norms in the 3D, so some restrictions on parameters are neccessary. In the defocusing case $\omega\gamma>0$ dissipative estimate in $H^2$ always hold, but there is an evidence that the $H^1$-norm may blow up in finite time in the self-focusing case $\omega\gamma<0$, see \cite{BRW} and references therein.
\par
On the other hand, the way how the dissipative estimate in the $H^2$-norm can be obtained is not essential for our main results. We just need the result of Theorem \ref{Th3.stand}. By this reason, we do not go further with derivation of this estimate and prefer to state it as an assumption.
\end{remark}
We now briefly discuss the so-called Man\'e projections of a global attractor and related finite-dimensional reduction, see \cite{Rob2} and references therein for more details.
\begin{definition} Let $\Cal A\subset H^2$ be the attractor for the solution semigroup $S(t)$. A linear projector $P: H^2\to \Cal V$, where $\Cal V$ is a finite-dimensional linear subspace of $H^2$, is a Man\'e projector if it is injective on the attractor $\Cal A$. Since $\Cal A$ is compact any Man\'e projector is a homeomorphism between $\Cal A$ and a finite-dimensional set $\widetilde{\Cal A}:=P\Cal A\subset\Cal V$. A Man\'e projector $P$ is called H\"older (resp. Lipschitz) Man\'e projector if its inverse $P^{-1}:\widetilde{\Cal A}\to\Cal A$ is H\"older (resp. Lipschitz) continuous.
\end{definition}
If the attractor $\Cal A$ possesses a Man\'e projector then the semigroup $S(t)$ acting on the attractor $\Cal A$ is topologically conjugate to the semigroup
$$
\widetilde S(t):=P\circ S(t)\circ P^{-1},\ \ \widetilde S(t):\widetilde{\Cal A}\to\widetilde{\Cal A},
$$
acting on a finite-dimensional compact set $\widetilde{\Cal A}\subset \Cal V$. In this sense the Man\'e projector realizes the finite-dimensional reduction of the limit dynamics on the attractor.
\par
Moreover, every Man\'e projector generates a system of ODEs for the limiting dynamics on the attractor -- the so-called inertial form. Namely, let $\widetilde\Psi(t):=P\Psi(t)$ where $\Psi(t)$ is a solution of a semilinear parabolic equation \eqref{4.eq}
 belonging to the attractor. Then, projecting the equation to $\Cal V$, we arrive at
\begin{equation}\label{4.IF}
\frac d{dt}\widetilde\Psi=P\circ A\circ P^{-1}(\widetilde \Psi)+PF(P^{-1}(\widetilde \Psi)).
\end{equation}
This IF has a specially nice form when $P$ is a spectral projector, i.e., when $PA=AP$. Namely,
\begin{equation}\label{4.IFs}
\frac d{dt}\widetilde \Psi=A\widetilde\Psi+PF(P^{-1}(\widetilde\Psi)).
\end{equation}
\begin{remark} As we have already mentioned in the introduction, Man\'e projections exist for more or less general abstract semilinear parabolic equations with global attractors. This fact is based on two general theorems. First of them claims that the attractor $\Cal A$ has finite box-counting dimension (standard corollary for a parabolic smoothing property for the equation of variations, see \cite{BV,Tem1,Rob1,Z}). And the second one is the so-called H\"older Man\'e theorem which claims that, for a compact set in a Banach space with finite box-counting dimension, H\"older Man\'e projectors are generic among all projectors on finite-dimensional planes with sufficiently large dimension, see \cite{Rob2}.
\par
However, this approach has essential drawbacks.
\par
First, we only know that "generic" projector is Man\'e without any algorithm to specify it in a concrete case. For instance, the spectral Man\'e projector may not exist, see counterexamples in \cite{EKZ,Z}.
\par
Second, the inertial form \eqref{4.IF} has only H\"older continuous vector field, so the uniqueness theorem may fail which would make this reduction to ODEs incomplete.
\par
Third, probably most important, as recent counterexamples show (see \cite{EKZ,Z}) the dynamics on the attractor $\Cal A$ with finite box-counting dimension may demonstrate clearly infinite-dimensional features, like limit cicles with super-exponential rate of attraction, travelling waves in Fourier space, etc. and the above scheme is unable to capture and distinguish them from "truely" finite-dimensional dynamics.
\par
As we will see, the situation is better when Lipschitz Man\'e projections are considered. However, they do not exist for general semilinear parabolic equations, so the case study is needed. Fortunately, our cross-diffusion system is exactly the exceptional case.
\end{remark}
The key result about Lipschitz Man\'e projectors is given by the so-called Romanov theory, see \cite{Rom1,Rom4,Z}.

\begin{theorem}\label{Th4.rom} Let $\Cal A$ be an attractor of the semilinear parabolic equation \eqref{4.eq}. Then the following conditions are equivalent:
\par
1. There exist a Lipschitz Man\'e projector.
\par
2. There is a spectral Lipschitz Man\'e projector.
\par
3. The solution semigroup $S(t):\Cal A\to\Cal A$ acting on the attractor can be extended to the Lipschitz continuous group $S(t)$, $t\in\R$, i.e., the inverse operators $S(-t)=S^{-1}(t)$ exist and Lipschitz continuous on the attractor.
\end{theorem}
We apply this theorem in order to get the main result of our paper.
\begin{theorem}\label{Th4.main} Let the assumptions of Theorem \ref{Th3.stand} hold and let, in addition, the cross diffusion coefficient $\omega\ne0$. Then the attractor $\Cal A$ of the solution semigroup $S(t)$ associated with equation \eqref{4.eq} possesses a spectral Man\'e projector. In particular, the limit dynamics on the attractor is described by a system of ODEs of the form \eqref{4.IFs} with Lipschitz continuous vector field.
\end{theorem}
\begin{proof} We will check condition 3 of the above theorem. The fact that $S(t):\Cal A\to\Cal A$ can be extended to a  group of homeomorphisms follows from the backward uniqueness theorem for semilinear parabolic equations (for instance, via the logarithmic convexity arguments, see e.g., \cite{BV,Z}), so we only need to check Lipschitz continuity.
\par
Let $\Psi_0^1,\Psi_0^2\in\Cal A$ be two points on the attractor. Since $\Cal A$ is invariant, there are two complete bounded trajectories $\Psi^i\in C_b(\R,H^2)$, $i=1,2$, belonging to the attractor such that $\Psi^i(0)=\Psi^i_0$, $i=1,2$. Let $v(t):=\Psi^1(t)-\Psi^2(t)$. Then this function satisfies the equation of variation for \eqref{4.eq} which has the form of \eqref{3.lincgl1}. Moreover, since the attractor is smooth, assumption \eqref{3.bounded} is satisfied uniformly with respect to $\Psi_0^1,\Psi_0^2\in\Cal A$. Thus, due to Corollary \ref{Cor1.main}, there are constants $N$, $C_N$ and $\theta_N$ (which are also uniform with respect to $\Psi_0^1,\Psi_0^2\in\Cal A$) such that
$$
\|v(t)\|_H\le C_Ne^{-\theta_N t}\|Pv(0)\|_H\le C_Ne^{-\theta_N t}\|v(0)\|_{H^2}.
$$
Applying the parabolic smoothing property to equation \eqref{3.lincgl1}, we finally arrive at
$$
\|v(t)\|_{H^2}\le C_Ne^{-\theta_N t}\|Pv(0)\|_H\le C_Ne^{-\theta_N t}\|v(0)\|_{H^2}.
$$
Thus, the desired backward Lipschitz continuity is verified and the theorem is proved.
\end{proof}
We now consider the particular case of the classical 3D cGL equation.

\begin{corollary} Let us consider the equation
\begin{equation}\label{4.cgll}
\Dt\Psi=(1+i\omega)\Dx\Psi+(1+i\beta)\Psi-(1+i\delta)\Psi|\Psi|^2.
\end{equation}
with periodic boundary conditions.
Assume that $\omega\ne0$ and any solution $\Psi(t)$ of this equation with $\Psi_0\in H^1$ exists globally in time $t\ge0$ (i.e., there is no finite time blow up of the $H^1$-norm). Then this equation possesses a smooth global attractor with spectral Lipschitz Man\'e projector.
\end{corollary}
\begin{proof}
Indeed, absence of finite time blow up for the $H^1$-norm together with dissipative $H$-estimate \eqref{4.hdis}  imply in a standard way the dissipativity in $H^1$, see e.g. \cite{Tem1}. In turn, since the classical cGL is subcritical in $H^1$, the dissipativity in $H^1$ implies the dissipative estimate \eqref{4.dis} and finishes the proof of the corollary.
\end{proof}
\begin{remark}As we have already mentioned, the condition $\omega\ne0$ is crucial for Theorem \ref{Th4.main}. Moreover, we expect that the theorem by itself is not true for $\omega=0$. Indeed, an explicit  counterexample of equation \eqref{4.eq} with $\omega=0$ and without normally hyperbolic inertial manifold has been constructed in \cite{Rom3}. On the other hand, the most difficult part in constructing counterexamples to existence of Lipschitz or log-Lipschitz Man\'e projections is {\it exactly} to break normal hyperbolicity, see \cite{EKZ,Z}, so we expect that the corresponding counterexample can be constructed by perturbing properly the example in \cite{Rom3}. We return to this somewhere else.
\end{remark}

\end{document}